\documentclass[12pt]{amsart}
\usepackage{amsmath}
\usepackage{color}
\usepackage[pdfborder={0 0 0},colorlinks=true]{hyperref}

\newtheorem{theor}{\indent\sc Theorem}[section]
\newtheorem{corol}[theor]{\indent\sc Corollary}
\newtheorem{lemma}[theor]{\indent\sc Lemma}   
\newtheorem{prop}[theor]{\indent\sc Proposition}

\newtheorem{dfn}[theor]{\indent\sc Definition}

\newtheorem{example}{\indent\sc {Example}}
\newtheorem{obs}{\indent\sc {Observation}}

\newcommand{\real}{{\mathbb R}}
\newcommand{\lpr}[2]{\langle{#1},{#2}\rangle}
\newcommand{\mink}{{\real}_{1}^{4}} 
\newcommand{\complex}{\mathbb{C}}

\newcommand{\hyper}{\mathbb{H}^{2}}

\newcommand{\euclidean}{{\mathbb{E}}^{3}}

\newcommand {\M}{{\mathcal M}}

\newcommand {\A}{{\mathcal A}}

\newcommand {\Holom}{{\mathcal H}}

\title{Spacelike minimal surfaces in $\mathbb R^4_1$  through of a $\theta$-Family}
\author{ M. P. Dussan$^{a}$,  \ A. P. Franco Filho$^a$ and R. S. Santos$^b$} 

\begin{document}
\maketitle

\centerline{\small{$^{A}$ Departamento de Matem\'atica - IME, Universidade de S\~ao Paulo, SP., Brazil}}
\centerline{\small{$^{B}$ CETENS, Universidade Federal do Rec\^oncavo da Bahia, Brazil}}
\centerline{Corresponding author: dussan@ime.usp.br (M.P. Dussan)}

\begin{abstract}
In this paper we introduce a $\theta$-family of spacelike surfaces in the Lorentz-Minkowski space $\mathbb R^4_1$ based in two complex valued functions $a(w), \mu(w)$, which when they are holomorphic we will be dealing with a family of spacelike minimal surfaces.   The $\theta$-family is such that it connects spacelike minimal surfaces in $\mathbb R^3_1$ to  spacelike minimal surfaces in $\mathbb R^3$. We study the family through of the curvature and we prove  that the family preserves planar points and moreover, that the existence of planar points corresponds to the existence of solutions of equation $|a_w(w)|^2 =0$. We also show that if a pair of  surfaces are associated through of a $\theta$-family then they can not be complete surfaces. As applications we focus to one type of  graph surfaces  in $\mathbb R^4_1$ and we prove that if the imaginary part of $a(w)$ is zero at least in a point then the surface cannot assume local representations of that type of graph. Several explicit examples are given. 

\end{abstract}

\vspace{0.2cm}

\vspace{0.05cm}
\section{Introduction} 
The minimal Riemannian surfaces have being extensively studied through of the times since that they are also strongly related to other mathematica's areas, for instance, to analysis, PDE and integrable systems. One of the most famous results in the theory of the  minimal surfaces is the Bernstein Theorem which says that a minimal surface in the 3-dimensional Euclidean space $\mathbb E^3$ which is the graph of an entire function must be a plane. After that result several authors have provided extensions of the Bernstein Theorem to other ambient spaces using different analytical or geometric techniques.

 This paper comes as an application of a recently published paper of the authors \cite{DPS} where using complex variables it was investigated the natural problem of finding spacelike graphs in the Lorentz-Minkowski space $\mathbb R^4_1$ which are minimal, in the sense that the mean curvature vector vanishes everywhere. Because the inner product in the ambient space is indefinite and the surface is positive definite, when considering graphs, there are two types, one defined over a timelike plane and one defined over a spacelike plane. Then using the Osserman Theorem (\cite{3}, \cite{7}) and Nitsche (\cite{jccn65}), the authors find entire solutions for the equations of two types of minimal graphs. In particular, for each type of graph the authors find entire minimal graphs in $\mathbb R^4_1$ which are not flat, showing that a Bernstein-type theorem cannot hold for either type of graph. For instance  Theorem 5.10 in \cite{DPS} establishes  a version of the Bernstein Theorem for  graphs called of second type and Example 7 in \cite{DPS} gives an explicit example of  entire not flat minimal graphs of second type. Here we pay attention that the process of construction of that example was  the light for we obtain the results of the current paper.
In fact, we figured out the possibility of constructing not just one surface but a 1-parameter family of minimal surfaces in $\mathbb R^4_1$ which links a minimal spacelike surface of $\mathbb R^3_1$ to a minimal spacelike surface of $\mathbb R^3$. We call that family of a $\theta$-family of spacelike surfaces in $\mathbb R^4_1$.

The main goal of this paper is then to contribute in the study spacelike minimal surfaces in $\mathbb R^4_1$ via a Weierstrass representation formula involving two holomorphic functions $a(w)$ and $\mu(w)$ written in isothermal coordinates (Definition (\ref{dfnfamily})). 
 That representation is a 1-parameter and it describes locally a $\theta$-family of spacelike minimal surfaces in  $\mathbb R^4_1$ with the technical features that when  $\theta =0$ and $\theta = \pi$ it will correspond to a pair of surfaces in the spaces $\mathbb R^3_1$ and $\mathbb R^3$  respectively.  Briefly, we call  that pair by a {\it pair associated}. Then we study locally the surfaces belonging to the family and we prove that the $\theta$-family transport minimal surfaces from  $\mathbb R^3_1$ to $\mathbb R^3$ preserving planar points. Moreover, using the Osserman Theorem we prove that if  $(M,X)$ in $\mathbb R^3_1$ and $(M,Y)$ in $\mathbb E^3$ are two minimal surfaces associated  by a $\theta$-family then the both surfaces are not complete (Corollary \ref{9}).  In order to do that, we look for the Gauss curvature formula for the $\theta$-family and we show that the existence of planar points depends directly from the existence of solutions of equation $|a_w|^2=0$, where $a_w$ means the complex derivate of the function $a(w)$.  In particular, these last facts   allow us also to characterize the family's surfaces when at least one of them is totally flat (Lemma \ref{lema3.6}).   Moreover we also find that the complex PDE system (\ref{eqassociadas}) has to be satisfied by the components of two surfaces to conclude that they form a pair associated through of a $\theta$-family.  Some examples are treated. 
 
As application, in section 4,  we focus in spacelike surfaces in $\mathbb R^3_1$ to find a necessary condition for the existence of  isothermal coordinates in which its Weierstrass representation assumes the way of graph surface of type $(x,y, A(x,y))$ where $A(x,y)$ is a smooth function a real-value.  In that case, we show that if the imaginary part of the holomorphic function $a(w)$ is equal to zero at least for one point then there are not such local isothermal coordinates, and hence the surface will not represent a graph surface of that type.

\section{Preliminaries}
\subsection{Spacelike Surfaces $S$ in $\mink$ and the Riemann surface associated to $S$} 
\begin{dfn}\label{dfn1.5}
A spacelike surface $S \subset \mink$ is a smooth $2$-dimensional submanifold of the topological real vector space $\real^{4}$ that 
at each point $p \in S$ its tangent plane $T_{p}S$ relative to the Lorentz product of $\mink$ is a spacelike plane. 

A parametric spacelike surface of $\mink$ is a two parameter map $(U,X)$ from a connected open subset $U \subset \real^{2}$ into 
$\mink$ such that the topological subspace $X(U)$ is a spacelike surface. 

We will assume, always, that $(X(U),X^{-1})$ is a chart of a complete atlas for a spacelike surface $S$ of $\mink$.    
\end{dfn} 

Let $U$ be a connected and simply connected open subset of the Euclidean plane $\real^{2}$ and $(x,y) \in U$.
If $X(x,y) = (X^{0}(x,y), X^{1}(x,y), X^{2}(x,y), X^{3}(x,y))$ is a spacelike parametric surface of $\mink$  we have 
a metric tensor induced by the Lorentzian product of $\mink$ given by 
\begin{equation}\label{metricag}
\mathbf{g} = \sum_{i,j} \lpr{D_{i} X}{D_{j} X} dx^{i} \otimes dx^{j},
\end{equation}
and the second quadratic form of $S = X(U)$ is the  quadratic symmetric $2$-form 
$B = \sum_{i,j} \Psi_{ij} dx^{i} \otimes dx^{j},$
which is obtained from of the covariant partial derivative given by the formula
\begin{equation}\label{derivparcialcov}
D_{ij} X - \sum_k \Gamma_{ij}^{k} D_{k} X = \Psi_{ij}.
\end{equation}
From the Christoffel symbols formula 
$\Gamma_{ij}^{k} = \frac{1}{2}g^{mk}\left(\frac{\partial g_{im}}{\partial x^j} + \frac{\partial g_{jm}}{\partial x^i} -\frac{\partial g_{ij}}{\partial x^m}\right),$
it follows that $\lpr{\Psi_{ij}}{D_{k} X} \equiv 0$ in all poins of $S$. We take the pointwise orthonormal basis for the normal bundle $NS$ given by $\tau(x,y)$ and $\nu(x,y)$, where 
$\tau(x,y)$ is a unit future directed timelike vector and $\nu(x,y)$ is a unit spacelike vector such that  
$\lpr{\nu(x,y)}{\partial_0} \equiv 0$. One has 
\begin{equation}
\Psi_{ij} = h_{ij} \tau + n_{ij} \nu,
\end{equation}
where by definition 
$h_{ij} = - \lpr{D_{ij} X}{\tau}$ and $n_{ij} = \lpr{D_{ij} X}{\nu}.$    

Since $\dim (N_{p}S) = 2$, the normal connection for $S$ is given by a covariant vector 
$\gamma = \displaystyle\sum_k \gamma_{k} dx^{k}$ with 
$\gamma_{k} = \lpr{D_{k} \tau}{\nu} = \lpr{D_{k} \nu}{\tau}.$ 
Then the set of structural equations for $S$ are given by 
\begin{align}
D_{ij} X &= \sum_k \Gamma_{ij}^{k} D_{k} X + h_{ij} \tau + n_{ij} \nu \label{eqgauss} \\
D_{k} \tau &= \sum_m h_{m}^{k} D_{m} X + \gamma_{k} \nu \label{eqweing1} \\
D_{k} \nu &= -\sum_m n_{m}^{k} D_{m} X + \gamma_{k} \tau, \label{eqweing2}
\end{align}
where $h_{m}^{k}g_{mj} = h_{kj}$ and  $n_{m}^{k}g_{mj} = n_{kj}$. We note that Equation (\ref{eqgauss}) corresponds to the Gauss equation and equations (\ref{eqweing1}) and (\ref{eqweing2}) correspond to the Weingarten equations.

\begin{dfn}\label{dfnminsurf}
The surface $S = X(U)$ is a minimal surface if and only if in all points of $S$,  the mean curvature vector $H_S = \frac{1}{2}\sum_{i,j} \Psi_{ij} g^{ij} $ is such that
$H_{S}= 0.$
 
 In particular, it follows from (\ref{derivparcialcov}) that an equivalent definition for minimal surface is 
$$2 H_{S} = \sum_{i,j} g^{ij} (D_{ij} X - \sum_k \Gamma_{ij}^{k} D_{k} X) = 
(\Delta_{\mathbf{g}} X^{0}, \Delta_{\mathbf{g}} X^{1}, \Delta_{\mathbf{g}} X^{2}, \Delta_{\mathbf{g}} X^{3}) = 0,$$ 
where $\Delta_{\mathbf{g}}$ is the Laplace-Beltrami operator over $S = (X(U), \mathbf{g})$.   
\end{dfn}

From the well-known theorem (see \cite{chern} for instance) which establishes that all spacelike surface admits 
an isothermal coordinate atlas, it follows that there exists a parametrization 
$$f(w) = (f^{0}(w), f^{1}(w), f^{2}(w), f^{3}(w)),  \; \; w = u + i v \in U' \subset \complex,$$
such that $f(U') \subset S = X(U)$, whose induced metric tensor is given by $\mathbf{g} = \lambda^{2} dw d\overline{w}$. That means that the vectors $f_u$ and $f_v$ satisfy
$$\lpr{f_{u}}{f_{u}} = \lambda^{2} = \lpr{f_{v}}{f_{v}} \; \; \; \mbox{ and } \; \; \; \lpr{f_{u}}{f_{v}} = 0.$$  We remember that in this case  the parametrization is also known  as a $\lambda$-{\it isothermal parametrization}. 

Now we extend the bilinear form of $\mink$ to a complex bilinear form over $\complex^{4} \equiv 
\real^{4} + i \real^{4}$ as follows
$$\lpr{\bf{x} + i \bf{y}}{\bf{a} + i \bf{b}} = \lpr{\bf{x}}{\bf{a}} - \lpr{\bf{y}}{\bf{b}} + i(\lpr{\bf{x}}{\bf{b}} + \lpr{\bf{y}}{\bf{a}},$$ 
whenever $\bf{x} + i \bf{y}, \bf{a} + i \bf{b}\in\complex^4$. Then for the complex derivate $f_{w} = \frac{1}{2} (f_{u} - i f_{v})$, we obtain
the equations 
\begin{equation}\label{eqmonge}
\lpr{f_{w}}{f_{w}} = 0 \; \; \mbox{ and } \; \; \lpr{f_{w}}{\overline{f_{w}}} = \lpr{f_{w}}{f_{\overline{w}}} = \lambda^{2}/2,
\end{equation}
which are  known as {\it Monge's equation}.

\vspace{0.2cm}
 If we have two isothermal charts $(U',f)$ and $(V,h)$ for $S$ then, when does it make sense, the overlapping map 
will be a holomorphic function, so one can see $M \doteq (U,\A)$ as a Riemann surface equipped with the conformal atlas $\A$ and with  the induced metric tensor  $ds^{2} = \lambda^{2}(w) \vert dw \vert^{2}$ as a compatible metric for the Riemann surface $M$. 

\vspace{0.1cm}
Finally, as we are assuming that $M$ is a 
connected and simply connected Riemann surface and since it does not exist compact spacelike surfaces in $\mink,$ 
 from now $M$ will be either 
the disk  \ $D = \{z \in \complex : z \overline z < 1\}$ which is a hyperbolic Riemann surface 
or the complex plane $\complex$, that is a parabolic Riemann surface.

\subsection{A solution for the Monge's equation and the Weierstrass representation}  Expanding in coordinates the first equation in   (\ref{eqmonge}) one has  
$$-(f_w^{0})^{2} + (f_w^{1})^{2} + (f_w^{2})^{2} + (f_w^{3})^{2} = 0.$$ 
Denoting the complex derivative of the components $f_w^i$ by $Z^i$ and assuming that the complex $Z^{1} - i Z^{2} \neq 0,$ it follows  
$$\frac{Z^{0} - Z^{3}}{Z^{1} - i Z^{2}} \; \frac{Z^{0} + Z^{3}}{Z^{1} - i Z^{2}} = \frac{Z^{1} + i Z^{2}}{Z^{1} - i Z^{2}}.$$ 
Hence defining 
$$a = \frac{Z^{0} + Z^{3}}{Z^{1} - i Z^{2}}, \; \;  \; \; \ \  b = \frac{Z^{0} - Z^{3}}{Z^{1} - i Z^{2}}  \; \; \ \ 
\mbox{ and } \; \; \ \ \mu = \frac{Z^{1} - i Z^{2}}{2},$$ 
 we can write a solution of the Monge's equation as  
\begin{equation}\label{15}
f_{w} = \mu W(a,b) \; \; \; \mbox{ where } \; \; \; W(a,b) = (a + b, 1 + ab, i(1 - ab), a - b),
\end{equation} 
for a pair of maps $(a,b) \in \mathcal F(M,\complex) \times \mathcal F(M,\complex)$ in the space $\complex^{4}$, where 
 $\mathcal F(U,\complex)$ denotes the set of smooth maps from $U$ to $\complex$, and where $\mu$ is a complex valued function.  Since from second equation of (\ref{eqmonge}) one has that 
$\lambda^{2} = 2 \lpr{f_{w}}{\overline{\; f_{w}}\;} = 4 \mu \overline{\mu} (1 - a \overline{b})(1 - \overline{a} b)$, it follows that  $\mu \neq 0$ and $1 - a \overline{b} \neq 0$ are needed conditions to obtain a surface without 
singularities in its metric, it just as we need it.

\vspace{0.2cm}

So, remembering that a function $f:U \to \mathbb C$ is called holomorphic if $f_{\overline w} =0$, where $f_{\overline w} = \frac{1}{2} (f_{u} + i f_{v})$,  it follows immediately for a minimal surface  the following result.

\begin{lemma}\label{lema1.7}
For a $\lambda$-isothermic spacelike surface $(U,f)$ the following statements are equivalent: 
\begin{itemize}
\item[(i)] The surface $f(U)$ is minimal, $H_{f}(w) \equiv 0$. 

\item[(ii)] The maps $\mu,a,b$ are holomorphic functions from $U$ into $\complex$.  
\end{itemize}
\end{lemma}
\begin{proof}
The proof follows easily applying Definition \ref{dfnminsurf} and using the Laplace-Beltrami operator on each component $f^i(w)$. Indeed, one gets $\Delta_{M} f^{i}(w) = \frac{2}{\lambda^{2}} (f^{i}(w))_{w \overline{w}} = 0$ for $i = 0,1,2,3,$ hence using the explicit expressions of $f^i_w$ of equation (\ref{15}) in terms of $\mu, a, b$ it follows the statement.
\end{proof}
 
Let us to focus in the Weierstass representation for the surface. Let $(U,X)$ be a spacelike parametric  surface of $\mink$ where 
$U \subset \real^{2}$ is a simply connected domain and $$X(x,y) = (X^{0}(x,y), X^{1}(x,y), X^{2}(x,y), X^{3}(x,y)).$$ Then, the vector $1$-form given by 
$$dX = \frac{\partial X}{\partial x} dx + \frac{\partial X}{\partial y} dy$$ 
is exact, therefore it is closed, and so we write the integral representation associated to $(U,X)$ as
\begin{equation}\label{eqint}
X(x,y) = X(x_{0},y_{0}) + \int_{(x_{0},y_{0})}^{(x,y)} \frac{\partial X}{\partial x} dx + \frac{\partial X}{\partial y} dy. 
\end{equation}

\vspace{0.1cm}
In sense converse,  one can also see that each solution of equation (\ref{eqint}) will represent a spacelike parametric surface $(U,X)$  whenever
$$E=\lpr{X_x}{X_x}>0, \ \ G = \lpr{X_y}{X_y}>0, \ \ F=\lpr{X_x}{X_y} \ \quad\mbox{and}\quad \ EG-F^2>0.$$

\vspace{0.1cm}
So from Definition \ref{dfnminsurf} and Lemma \ref{lema1.7}, we obtain the next corollary.
\begin{corol}
Let $U\subset\real^2$ a simply connected domain. If $(U,X)$ is a spacelike minimal  parametric surface which is solution 
of the integral equation (\ref{eqint}), then each coordinate function of $X(x,y)$ is a harmonic real-valued function on $U.$ 
\end{corol} 
\begin{proof}
Indeed, the Laplace-Beltrami operator $\Delta_{M}$ is a tensorial operator, defined by contraction of the Gauss equation (\ref{eqgauss}), 
as follows in Definition \ref{dfnminsurf}. 
\end{proof}
We note that in isothermal coordinates the integral representation (\ref{eqint}) is known as the
Weierstrass integral representation, which in terms of $\mu, a, b$, can be written as 
\begin{equation}\label{23}
f(w) = p_{0} + 2 \Re \int_{w_{0}}^{w} \mu(\xi) W(a(\xi),b(\xi)) d \xi,
\end{equation}
where $f_w(w)$ is solution of the Monge's equation (\ref{eqmonge}).

\subsection{The structure equations and the equation $\tau_{w} = \eta f_{\overline{w}}$} 

In this subsection we will assume that $(U,f)$ is a $\lambda$-isothermal parametric 
surface of $\real^3_1$ with normal Gauss map given by $\tau : U \to \hyper$ and we will obtain the structure equations and the Gauss curvature for that kind of surfaces in terms of the functions $\mu(w)$ and $a(w)$.
 
 \vspace{0.2cm}
First we note that since $\lpr{\tau}{\tau} = -1$ it follows that $\lpr{\tau_{w}}{\tau} = 0$. Now since $\lpr{\tau}{f_{w}} = 0,$  there exists 
functions $\eta(w)$, $h(w)$ such that for $w \in U$,
$$\tau_{w}(w) = h(w) f_{w}(w) + \eta(w) f_{\overline{w}}(w).$$  In this case, since $\nu = (0,0,0,1)$, the Weingarten equation  (\ref{eqweing1}) becomes
$$
\begin{cases} 
\tau_{u} = (\frac{h_{11}}{\lambda^{2}}) \; f_{u} + (\frac{h_{12}}{\lambda^{2}}) \; f_{v} \\ 
\tau_{v} = (\frac{h_{12}}{\lambda^{2}}) \; f_{u} + (\frac{h_{22}}{\lambda^{2}}) \; f_{v}.
\end{cases}$$ 

\begin{prop}\label{prop3.2}
If $(U,f)$ is a minimal $E(f)$-isothermal parametric surface of $\real^3_1$ with normal Gauss map 
$\tau : U \to \hyper,$ then the Weingarten equation for the surface is given by 
\begin{equation}\label{eq24}
\tau_{w}(w) = \eta(w) f_{\overline{w}}(w) \; \; \mbox{ where } \; \; \eta(w) = \sqrt{K_{f}(w)\;} e^{i \psi(w)},   
\end{equation} 
and $K_{f}(w)$ is its Gauss curvature. 
\end{prop}
\begin{proof}
Since the mean curvature of $S$ is null, we have two maps $p(w), q(w)$ defined by $p(w) = h_{1}^{1}(w) = - h_{2}^{2}(w)$ 
and $q(w) = h_{1}^{2}(w) = h_{2}^{1}(w)$ such that $K_{f}(w) = p^{2}(w) + q^{2}(w)$. Now, since $f_{ w} = \frac{1}{2}(f_u - i f_v)$ it follows  the identities 
$f_{u} = f_{w} + f_{\overline{w}}$ and  $f_{v} = i(f_{w} - f_{\overline{w}})$, which imply  immediately 
$$\tau_{w} = \frac{\tau_{u} - i \tau_{v}}{2} = p(f_{w} + f_{\overline{w}}) - i q(f_{w} - f_{\overline{w}}) = (p - iq) f_{\overline{w}}.$$ 
Now we define $\eta = p - iq.$ Then since $\eta \overline{\eta} = K_{f}$, one obtains the equation (\ref{eq24}) when taking the polar form for the map $\eta$.  
\end{proof}
Next it denotes with $Pl(S)$ the set of planar points of the surface, that means for $(U,f)$  a parametric surface that set is ${Pl(S)} = \{w \in U: K_{f}(w) = 0\}$.

\begin{corol}\label{12}
Let $(U,f)$ be a minimal $E(f)$-isothermal parametric surface of $\real^3_1$ with normal Gauss map 
$\tau : U \to \hyper$. If $\mbox{Pl(S)} = \emptyset,$ then the Riemann surface associated $M = (U,\A)$ is conformally equivalent to the disk 
$D = \{z \in \complex : z \overline{z} < 1\}$. 
\end{corol}

\begin{proof}
Indeed, if $Pl(S) = \emptyset$ then the Riemann surface associated $M = (U,\A)$ has points where $K_f(w)\neq 0,$ thus it cannot be whole the plane.
\end{proof}

\begin{obs}
Here we note that from $\tau_{w} = \eta f_{\overline{w}}$ it follows that the functions $E(\tau)$ and $E(f)$ are related by 
$$\frac{1}{2} E(\tau) = \lpr{\tau_{w}}{\tau_{\overline{w}}} = \frac{1}{2} K_f(w) \lpr{f_{w}}{f_{\overline{w}}} = \frac{1}{2} K_f(w) E(f).$$ 
This relation allows us to obtain a formula for the Gauss curvature of $f(U)$, namely,  
$$K(f) = \frac{\vert a' \vert^{2}}{\mu \overline{\mu} (1 - a \overline{a})^{2}}\overset{.}{}$$ 
\end{obs}

Now we write the Gauss equation (\ref{eqgauss}) for the immersion $(M,f)$ in $\real^3_1$ by  
$$f_{ww} = A f_{w} + B f_{\overline{w}} + \Omega \tau,$$ 
with $A,B,\Omega\in C^\infty(U,\real).$
Then, using that the immersion is minimal with $\lpr{f_{w}}{f_{w}} = 0$, $\lpr{f_{w}}{f_{\overline{w}}} = \frac{\lambda^{2}}{2}$ and $\lpr{f_{ww}}{\tau} = - \lpr{f_{w}}{\tau_{w}} = - \lpr{f_{w}}{\eta f_{\overline{w}}} = - \Omega$, it follows that 
$$f_{ww} = 2 \frac{\lambda_{w}}{\lambda} f_{w} + \frac{\eta \lambda^{2}}{2} \tau.$$  
Again, since $(M,f)$ is a minimal surface we obtain $f_{ww\overline{w}} = 0$. Moreover, from $\lpr{\tau_{w\overline{w}}}{\tau} = 0$ 
one obtains that $\Omega_{\overline{w}} = 0$. 

\begin{prop}
Let $(M,f)$ be a minimal immersion from $M$ into $\real^3_1.$ Then for the map $\Omega = \frac{\eta \lambda^{2}}{2}$ it follows that
\begin{itemize}
\item[(1)] for each $w \in M$,  $\Omega(w) = 2\mu(w) a'(w)$  
\item[(2)]  $\Omega(w)$ is holomorphic.
\end{itemize}
\end{prop}
\begin{proof}
We are going to compute $\eta$. Firstly, one has  $\tau^{0}_{w} = \eta f^{0}_{\overline{w}}$ if and only if 
$$\frac{\partial}{\partial \overline{w}}\left( \frac{1 + a \overline{a}}{1 - a \overline{a}} \right)= 
\frac{2a_{w} \overline{a}}{(1 - a \overline{a})^{2}} = 2 \overline{\mu} \; \overline{a} \; \overline{\eta} \; \Leftrightarrow \; 
\eta = \frac{a_{w}}{\overline{\mu}(1 - a \overline{a})^{2}}\overset{.}{}$$ 

It follows from the latter and from $\lambda^2 = 4 \mu \bar \mu (1 -a \bar a)^2$ that 
$$\Omega = \frac{\eta \lambda^{2}}{2} = \frac{a_{w}}{2 \overline{\mu}(1 - a \overline{a})^{2}} \; 
4 \mu \overline{\mu} (1 - a \overline{a})^{2} = 2\mu a'.$$ 

Finally, the holomorphicity of $\Omega$ follows from the expression in item (1).  
\end{proof}

\vspace{0.2cm}
The following proposition and corollary are, respectively, the versions of the Proposition \ref{prop3.2} and Corollary \ref{12} for minimal surface of the space $\euclidean,$ which have analogous proofs.
 
\begin{prop}
If $(U,h)$ is a minimal $E(h)$-isothermal parametric surface of $\euclidean$ with normal Gauss map 
$\nu : U \to \mathbb{S}^{2},$ then the Weingarten equation for the surface is given by 
\begin{equation}\label{eq25}
\nu_{w}(w) = \xi(w)  h_{\overline{w}}(w) \; \; \mbox{ where } \; \; \xi(w) = \sqrt{-K_{h}(w)\;} e^{i \phi(w)},   
\end{equation} 
and $K_{h}(w)$ is its Gauss curvature. 
\end{prop}
\begin{corol}
Let $(U,h)$ be a minimal $E(h)$-isothermal parametric surface of $\euclidean$ with normal Gauss map 
$\nu : U \to \mathbb{S}^{2}$. If $\mbox{Pl(S)} = \emptyset,$ then the Riemann surface associated $M = (U,\A)$ is conformally equivalent to the disk 
$D = \{z \in \complex : z \overline{z} < 1\}.$
Moreover the Gauss curvature is given by  
$$K(h) = -\frac{\vert a' \vert^{2}}{\mu \overline{\mu} (1 + a \overline{a})^{2}}\overset{.}{}$$
\end{corol}

\section{The $\theta$-family of minimal surfaces in $\mink$} 

We begin taking in the Weierstrass integral representation (\ref{23}) the pair $(a(w), b(w))$, with $b(w) = e^{i\theta} a(w)$ and $a(w)$ being a holomorphic function, for defining  the $\theta$-family of minimal surfaces. We will study  the Gauss curvature and the planar points for the family's surfaces.

 In the next in this paper we denote by $\mathcal H(U, \mathbb C)$ the space of holomorphic functions from $U$ to $\mathbb C$. 

\begin{dfn}\label{dfnfamily}
Let $\mu,a\in\Holom(U,\complex)$. For each $\theta \in \real$ and $w\in U,$ the map 
\begin{eqnarray}
F:\real \times U &\longrightarrow &\mink\nonumber\\
(\theta;w) &\longmapsto & F(\theta;w) = P + 2 \Re \int_{w_{0}}^{w} \mu(\xi) W \left(a(\xi),e^{i\theta} a(\xi) \right) d \xi,  
\end{eqnarray} 
 where  $W ( a,e^{i\theta} a) = \left((1 + e^{i\theta})a, 1 + e^{i\theta} a^{2}, i(1 - e^{i\theta} a^{2}), (1 - e^{i\theta}) a \right)$ and $P \in \mink$ is an arbitrary point, 
determines a family of minimal isothermal parametric surfaces in $\mink,$ which we call a $\theta$-family of minimal surfaces in $\mink$.

\vspace{0.1cm}

In similar way, we define the conjugated $\theta$-family as being the family of minimal isothermal parametric surfaces in $\mathbb R^4_1$, given by
\begin{eqnarray}
H: \real \times U &\longrightarrow & \mink\nonumber \\
(\theta;w) &\longmapsto & H(\theta;w) = Q + 2 \Im  \int_{w_{0}}^{w} \mu(\xi) W\left(a(\xi),e^{i\theta} a(\xi) \right) d \xi,  
\end{eqnarray} 
where $Q\in\mink$ is an arbitrary point.  

\vspace{0.1cm}

The family given by $F(\theta;w) + i H(\theta;w),$ passing throuth the point $P + iQ \in 
\complex^{4},$  is called the $\theta$-family of minimal surfaces in $\complex^{4}.$
\end{dfn} 
\begin{obs}
A $\theta$-family is associated to a $\theta$-periodic hypersurface $F(\real;M)$ of $\mink,$ thus we can see the family 
as a proper immersion of $\real \times M$ in a hypersurface that is foliated by minimal surfaces. In particular, for $\theta = 0$ we have a (maximal) 
minimal surface in $\real_1^3 \equiv \{p \in \mink : \lpr{p}{(0,0,0,1)} = 0\}$ and for $\theta = \pi$ we have a minimal surface in 
$\euclidean \equiv \{p \in \mink : \lpr{p}{(1,0,0,0)} = 0\}.$ This last surface we call the associated translated surface.
\end{obs}
Next we use again  the representation (\ref{23}) to look for the partial differential equations which have to be satisfied by the pair of associated surfaces.  So, in particular one can to identify if given two minimal spacelike surfaces, one of them in $\mathbb R^3_1$ and the other one in  $\mathbb E^3$,  are associated through of a $\theta$-family. 
 
In fact, we represent in arbitrary coordinates $(x,y)$ the minimal parametric  surface in $\real^3_1$ as being
$X(x,y) = \left (X^{0}(x,y),X^{1}(x,y),X^{2}(x,y),0 \right)$. Identifying $(x^{1},x^{2}) = (x,y) \in U$, by Definition \ref{dfn1.5}, 
we have that the metric tensor over $U$ has components given by  
$$g_{ij} = - \frac{\partial X^{0}}{\partial x^{i}} \; \frac{\partial X^{0}}{\partial x^{j}} +  
\frac{\partial X^{1}}{\partial x^{i}} \; \frac{\partial X^{1}}{\partial x^{j}} + 
\frac{\partial X^{2}}{\partial x^{i}} \; \frac{\partial X^{2}}{\partial x^{j}}\overset{.}{}$$ 


Now, by one unique coordinate transformation, at each neighborhood of a given point, we obtain isothermal parameters 
$w = u + iv$ related to the two holomorphic functions $ a(w), \ \mu(w)$, in such way that  $X$ can be written as 
$$X(w) = \left (X^{0}(z(w)),X^{1}(z(w)),X^{2}(z(w)),0 \right) = P + 2\Re \int_{w_{0}}^{w} \mu(\xi) W(a(\xi),a(\xi)) d \xi,$$   
for $(U,X)$ in $\real^3_1$,  and such that the associated translated surface can be written, via the $\theta$-family,  as
$$Y(w) = (0, Y^{1}(z(w)),Y^{2}(z(w)),Y^{3}(z(w))) = Q + 2\Re \int_{w_{0}}^{w} \mu(\xi) W(a(\xi),-a(\xi)) d \xi,$$
for $(U,Y)$ in $\euclidean$. 

Now, since $W(a,a) = (2a, 1 + a^{2},i(1 - a^{2}),0)$ and $W(a,-a) = (0, 1 - a^{2}, i(1 + a^{2}), 2a)$, making some computations we obtain the partial differential equations in complex variable 
linking $X = X(x,y)$ to $Y = Y(x,y)$. We establish them in the next definition.
\begin{dfn}\label{dfn2.3}
For the pair of minimal parametric surfaces associated to a $\theta$-family, namely $(U,X)$, $(U,Y)$ where $X(U) \subset \real^3_1$ and 
$Y(U) \subset \euclidean,$ the relations 
\begin{equation}\label{eqassociadas}
\frac{\partial Y^{3}(w)}{\partial w} = \frac{\partial X^{0}(w)}{\partial w},\quad  \; \; \frac{\partial Y^{1}(w)}{\partial w} = -i \frac{\partial X^{2}(w)}{\partial w} 
\; \mbox{ and } \; \; \; \frac{\partial Y^{2}(w)}{\partial w} = i \frac{\partial X^{1}(w)}{\partial w}
\end{equation} 
are satisfied. We define equation (\ref{eqassociadas}) as the associated equations to the $\theta$-family. 
\end{dfn}

\subsection{The Gauss curvature to the $\theta$-family} We establish next a pointwise base for the normal bundle associated to the $\theta$-family which we will use to calculate the Gauss curvature of the family surfaces.

\vspace{0.2cm}
 Consider the following two lightlike vectors associated to the family,
\begin{align}
L_{3}(a) = \left(1 + a \overline{a}, a + \overline{a},-i(a - \overline{a}),-1 + a \overline{a} \right)\label{aplicacaoL3} \\
L_{0}(e^{i \theta} a) = \left(1 + a \overline{a}, ae^{i \theta} + \overline{a}e^{-i \theta},
-i(a e^{i \theta} - \overline{a}e^{i \theta}\right), 1 - a \overline{a} ), \label{aplicacaoL0}
\end{align}
where we have omitted the variable $w \in U$ in $a = a(w)$.

\begin{lemma}\label{lema2.4}
The correspondent normal vectors $\tau, \nu$ for the $\theta$-family, are explicitly given by
\begin{align}
\tau(\theta;a) = \frac{1}{ \sqrt{-2\lpr{L_{3}(a)}{L_{0}(e^{i \theta} a)}}} \;[ L_{3}(a) + L_{0}(e^{i \theta} a) ]\label{taulema2.4} \\ 
\nu(\theta;a) = \frac{1}{ \sqrt{-2\lpr{L_{3}(a)}{L_{0}(e^{i \theta} a)}}} \; [L_{3}(a) - L_{0}(e^{i \theta} a)],\label{nulema2.4}
\end{align} 

and the metric tensor for set of surfaces is given by 
\begin{equation}
\mathbf{g}(\theta;w) = 4 \mu(w) \overline{\mu(w)} \left(1 - 2 a(w) \overline{a(w)}  \cos \theta + (a(w)\overline{a(w)})^{2} \right) dw d\overline{w}.
\end{equation} 

Moreover,  $\tau^{3}(\theta;w) \equiv 0$ and $\nu^{0}(\theta;w) \equiv 0$ for all $(\theta;w)\in\real\times M.$ 
\end{lemma}

\begin{proof}
From $\lpr{L_{0}(a e^{i\theta})}{L_{3}(a)} = - 2 \vert 1 - a \overline{a e^{i\theta}} \vert^{2}$ it follows that 
$\tau$ and $\nu$ given in formulas (\ref{taulema2.4}) and (\ref{nulema2.4}) are unit and orthogonal vectors fields of the normal bundle $NS$. 
Since the components $L_{0}^{3}$ and $L_{3}^{3}$ are given by $L_{0}^{3} = 1 - a \overline{a}$ and  $L_{3}^{3} = -1 + a \overline{a},$
it follows that the component  $\tau^{3}(\theta;a(w)) = 0$ for each $\theta \in \real$ and $w \in U$. Finally, one sees that the vector
 fields (\ref{taulema2.4}) and (\ref{nulema2.4}) in fact define two Gauss maps. 
\end{proof}
Then we can assume the Minkowski frame associated to a $\theta$-family $F(\theta;a)$ as being
\begin{equation}\label{162}
\M(\theta;a) = \left \{\tau(\theta;a), \left(\frac{1}{\lambda} F_{u}\right)(\theta;a), \left(\frac{1}{\lambda} F_{v} \right)(\theta;a), 
\nu(\theta;a)\right \},
\end{equation}
where $\tau$ and $\nu$ are given by equations (\ref{taulema2.4}) and (\ref{nulema2.4}).

\vspace{0.2cm}
For next we recall that when the coordinates are $\lambda$-isothermal, that is $E = G = \lambda^{2}$ and $F = 0$ in each point of $X(U) = S$,  we have that the Gauss curvature $K$ is given by
\begin{equation}\label{formulacurvGauss}
K(w) = -\frac{1}{2 E} \Delta \ln E = -\frac{1}{\lambda^{2}} \Delta \ln \lambda.
\end{equation}
  
\begin{theor}\label{lema2.5}
There exists a $\theta$-family of isothermal parametric surfaces $F(\theta;w)$ with domain in a Riemann surface $M = (U,\A)$ 
if and only if, the timelike normal vector $\tau$ in the Minkowski frame $\mathcal{M}(\theta;w)$ satisfies $\lpr{\tau(\theta;w)}{\partial_{3}} = 0$
 for each 
$\theta \in \real$ and $w \in U$.
Moreover, the latter happens if, and only if $a(w)$ and $\mu(w)$ are holomorphic functions from $U$ in $\complex$. 

Furthermore, the existence of planar points corresponds to the existence of zeros of the function $a' \doteq a_{w}$ and the Gauss curvature 
for the family is given by
\begin{equation}\label{curvGausslema2.5}
K(F(\theta;w)) = \frac{\vert a' \vert^{2} \left ( (1 + \vert a \vert^{4}) \cos \theta - 2 \vert a \vert^{2}\right)}
{\vert \mu \vert^{2} \left(1 - 2 \vert a \vert^{2} \cos \theta  + \vert a \vert^{4}\right)^{3}}\overset{.}{}
\end{equation}   
\end{theor} 

\begin{proof}
We will prove that if $a$ is a holomorphic function then $\mu$ is a holomorphic function. For that, we 
need to see that $\tau^{3}(\theta;w) \equiv 0$  implies that $a$ is a holomorphic function, so then we can define the $\theta$-family as required.

First, we retake the partial equation obtained in \cite{DuFF} which has to be satisfied by the  integration factor $\mu$, namely, 
$$\frac{\partial}{\partial \overline{w}} \mbox{Log } \mu = 
\frac{\overline{a} b_{\overline{w}}}{1 - \overline{a}b} + \frac{\overline{b} a_{\overline{w}}}{1 - a \overline{b}}\overset{.}{}$$
Therefore, since for all $w \in U$,  $f_{w\overline{w}}(w)$  is a real normal vector of $\mink$ 
we obtain that, for $b = a e^{i \theta}$, 
$$\frac{\partial}{\partial \overline{w}} \mbox{Log } \mu = 
\frac{\overline{a} a_{\overline{w}}}{1 - 2 \cos \theta \vert a \vert^{2} + \vert a \vert^{4}}
(1 - e^{-i \theta} a \overline{a}) + (1 - e^{i \theta} a \overline{a}),$$ 
or equivalently 
\begin{equation}\label{1}
\frac{\partial}{\partial \overline{w}} \mbox{Log } \mu = 
\frac{2\overline{a} a_{\overline{w}}\left (1 - \vert a \vert^{2} \cos \theta \right)}{1 - 2 \vert a \vert^{2}  \cos \theta + \vert a \vert^{4}}\overset{.}{}
\end{equation}

\vspace{0.1cm}
On the other hand, since the map $\mu$ does not depend of $\theta$ it follows from equation (\ref{1}),  that $a_{\overline{w}} = 0$ and so $a$ is holomorphic function. Then, again by equation (\ref{1}), one concludes that $\mu$ has to be also holomorphic function. Thus one can define the $\theta$-family.

\vspace{0.1cm}
Next we will compute the Gauss curvature of each surface of the family using equation (\ref{formulacurvGauss}). Since in our case we have that 
$\Delta\ln E = 4(\ln[4\mu\overline{\mu}(1 - 2 a \overline{a} \cos \theta  + (a\overline{a})^{2})])_{w\overline{w}},$ and  since 
\begin{eqnarray*}
\left (\ln[4\mu\overline{\mu}(1 - 2 a \overline{a} \cos \theta + (a\overline{a})^{2})]\right)_{w\overline{w}} &=& \left(\frac{\mu_w}{\mu}+
2a_w\left(\frac{a \overline{a}^{2} - \overline{a} \cos \theta}
{1 - 2 \vert a \vert^{2}  \cos \theta + \vert a \vert^{4}}\right)\right)_{\overline{w}}\\
&=& \frac{2|a'|^2 \left (2\vert a\vert^2 - (1+\vert a\vert^4)\cos\theta\right)}{(1 - 2 \vert a \vert^{2} \cos \theta + \vert a \vert^{4})^{2}},
\end{eqnarray*}
one obtains the formula (\ref{curvGausslema2.5}). It notes in particular that if $\theta = 0$, $K(F(0;w)) \geq 0$ and if $\theta = \pi$ 
then $K(F(\pi,w)) \leq 0$.  Finally, it follows from the equation (\ref{curvGausslema2.5}) that the set of planar points is given by $\{w \in U: a'(w) = 0\}$. 
\end{proof}

\begin{example}
If  $\mu(w) = 1$ and $a(w) = \sin(w)$ for $w\in U$, then taking $\theta = 0$ and $\theta = \pi$ we have minimal surfaces in $\real^3_1$ and in $\euclidean$ respectively, with a closed and discrete set of planar points for the booth surfaces. In particular, when $\theta = \pi$ the set is infinite
$ \{w \in U: \cos(w) = 0\}.$ More explicitly, assuming those values of $\mu(w)$ and $a(w)$ under the restriction $|\sin w|^2 \ne 1$, one obtains that the minimal spacelike surface in $\real^3_1$
is given by $$F(0, w) = \Re  \displaystyle ( -4 \cos w, 3w - \frac{\sin(2w)}{2},
i (w + \frac{\sin(2w)}{2}), 0)=
$$
$$
=(-4 \cos u \cosh v, 3u - \sin u \cos u (\cosh^2 v+ \sinh^2v), - v + \cosh v \sinh v (\sin^2 u - \cos^2 u), 0),
$$
 the associated translated surface in $\euclidean$ is given by
$$F(\pi, w) = \Re  \displaystyle ( 0, w + \frac{\sin(2w)}{2},
i (3w - \frac{\sin(2w)}{2}), -4 \cos w)=
$$
$$
=(0,  u + \sin u \cos u (\cosh^2 v+ \sinh^2v), - 3v - \cosh v \sinh v (\sin^2 u - \cos^2 u), -4 \cos u \cosh v),
$$
and both surfaces are associated by the $\theta$-family 
$$F(\theta, w) = 2 \Re  \int_{w_o}^w ((1+ e^{i\theta}) \sin \xi, 1+ e^{i\theta} \sin^2 \xi, i (1 - e^{i\theta} \sin^2\xi ), (1- e^{i\theta}) \sin \xi) d\xi.
$$
By formulas (\ref{taulema2.4}) and (\ref{nulema2.4}) in Lemma \ref{lema2.4}, the Gauss maps for $F(0, w)$ and $F(\pi, w)$ are given respectively by
$$
\tau(0, w) = \frac{1}{1- |\sin w|^2}(1+ |\sin w|^2, 2 \Re(\sin w), 2 \Im(\sin w), 0), \ \ \ \nu(0, w) = (0,0,0, -1).
$$
$$
\tau(\pi, w) = (1, 0, 0, 0), \ \ \ \nu(\pi, w) = \frac{1}{1+ |\sin w|^2}(0,  2 \Re(\sin w), 2 \Im(\sin w), -1+ |\sin w|^2).
$$

From formula (\ref{curvGausslema2.5}) the Gauss curvature for the two minimal surfaces are given by
$$
K(F(0, w)) = \frac{|\cos w|^2}{(1 - |\sin w|^2)^2}, \ \ \  \ K(F(\pi, w)) = - \frac{|\cos w|^2}{(1 + |\sin w|^2)^2}.
$$
Moreover, the Gauss curvature in the point $w$ belonging to the surface $F(\theta, w)$, with $\theta$ fixed, is given by
$$
K(F(\theta, w)) = \frac{|\cos w|^2 [(1+ |\sin w|^4) \cos \theta - 2 |\sin w|^2]}{(1 - 2 |\sin w|^2 \cos \theta + |\sin w|^4)^3}.
$$
\end{example}

\vspace{0.1cm}
\begin{example}
Let  $\mu(w) = e^{-w}$ and $a(w) = i e^w$ for  $w= u+iv$ with  $u>0$. Then the $\theta$-family is given by
$$F(\theta, w) = 2 \Re  \int_{w_o}^w  e^{-\xi} ((1+ e^{i\theta})ie^\xi, 1- e^{i\theta}e^{2\xi}, i (1 + e^{i\theta}e^{2\xi}), (1- e^{i\theta}) i e^\xi) d\xi = 
$$
$$
= - 2( u\sin \theta + v(1+\cos \theta), e^{-u} \cos v + e^{-v} \sin(\theta +u), e^{-u} \sin v - e^{-v} \cos (\theta +u), - u\sin \theta + v(1- \cos \theta)).
$$
The Gauss curvature for this family is never zero, because for all $(\theta, w$), 
$$
K(F(\theta, w)) =  e^{4u} \frac{(1+ e^{4u})\cos \theta - 2 e^{2u}}{(1-2 e^{2u} \cos \theta + e^{4u})^3},
$$
which never gets to be zero. Moreover,  the minimal spacelike surfaces in $\mathbb R^3_1$ and $\mathbb R^3$ associated to the family are given, respectively, by:
$$
 F(0, w) = -2 (2v, e^{-u} \cos v + e^{-v} \sin u, e^{-u} \sin v - e^{-v} \cos u, 0),
$$
$$
 F(\pi, w) = -2 (0, e^{-u} \cos v - e^{-v} \sin u, e^{-u} \sin v + e^{-v} \cos u, 2v),
$$
and for all $w$, $F(0, w)$ has positive Gauss curvature  unlike  $F(\pi, w)$ which has negative curvature in all its points.
\end{example}

\begin{example}
Let  $X_0(u,v) = (\sin u \cosh v, \sin u \sinh v, u, 0)$  a minimal spacelike isothermal surface in $\mathbb R^3_1$ and 
$X_{\pi} (u,v) = (0, v, -\cos u \cosh v, \sin u \cosh v)$ a minimal spacelike isothermal surface in $\mathbb R^3$.  Then one observes that the equations (\ref{eqassociadas}) are satisfied and hence there exists a $\theta$-family $F(\theta, w)$ which related the two surfaces. That $\theta$-family is given then  by equation (13) for  $$a(w) = \frac{i \cos w}{\sin w +1} \ \ \ {\rm and} \ \ \  \mu(w) = -\frac{i}{4}(\sin w +1),$$
where we have assumed that $w = u+iv$ is such that $\sin w \ne - 1$  and  $|a(w)| \ne 1$. More explicitly they are associated by the $\theta$-family 
$$F(\theta, w) = 2 \Re  \int_{w_o}^w  \displaystyle( (1+ e^{i\theta}) \cos \xi, -i(\sin \xi + 1- e^{i\theta}\frac{\cos^2 \xi}{\sin \xi +1}),  \sin \xi +1 + e^{i\theta} \frac{\cos^2 \xi}{\sin \xi +1}, (1- e^{i\theta}) \cos \xi) d\xi. 
$$
Now, from formula (\ref{curvGausslema2.5}) it follows that for all the surfaces of the family $F(\theta, w)$ there are not points in which the Gauss curvature is zero since $|a'(w)|^2 = \frac{1}{|\sin w +1|^2} \ne 0$ for all $w$.  
\end{example}

\subsection{The incompleteness of minimal surfaces} In this subsection we study the incompleteness of the associated pair $(U,X)$ and $(U,Y)$.

We start writing in components and in arbitrary parameters $(x,y)\in M$, the vector field 
$\tau(x,y) = (\tau^{0}(x,y), \tau^{1}(x,y), \tau^{2}(x,y), 0) \in \mathbb{H}^{2}.$ 
Now taking the hyperbolic stereographic projection for $\mathbb{H}^{2}$ onto the disk 
$\mathbb{D} =\{z \in \complex : z\overline{z} < 1 \}$ given by 
$$a(x,y) = (st_{h}(\tau(x,y)) = \left. \frac{\tau^{1} + i \tau^{2}}{\tau^{0} + 1}\right\vert_{(x,y)},$$ 
we see that the complex function $ a(x,y) \in \mathbb{D}$ is globally defined over $M$. 
At next, it defines the unique map $\nu(x,y)$ associated to $\tau(x,y)$ as
$$\nu(x,y) = \left. \frac{1}{\tau^{0}}(0, \tau^{1}, \tau^{2}, -1)\right\vert_{(x,y)},$$ 
and it considers the map $\Phi(\tau(x,y)) = \nu(x,y)$ from $\mathbb{H}^{2}$ onto the open South hemisphere $\mathbb{S}_{-}^{2}$ of the 
unit sphere of $\euclidean$. Then we obtain the following lemma. 

\begin{lemma}\label{lema2.7}
The map 
$$\Phi : \tau = (\tau^{0}, \tau^{1}, \tau^{2}, 0) \in \mathbb{H}^{2} \to
\nu = \frac{1}{\tau^{0}}(0, \tau^{1}, \tau^{2}, -1) \in \mathbb{S}_{-}^{2}$$ is a bi-holomorphism from the hyperbolic plane 
onto the South hemisphere of the Riemann sphere $\mathbb{S}^{2}$. 
\end{lemma} 

\begin{proof}
First we note that since $(\tau^{0})^{2} = 1 + (\tau^{1})^{2} + (\tau^{2})^{2}$ it follows that $\lpr{\nu}{\nu} = 1$. On the other hand, using the north pole stereographic projection $st_{north}$ from $\mathbb{S}^{2}$ onto the equatorial plane $\complex,$ it follows that the South hemisphere maps onto 
the disk $\mathbb{D}$. If $z \in \mathbb{D}$ then $z = (st_{north}) \circ \Phi \circ st_{h}(z).$   
\end{proof} 

From the Lemma \ref{lema2.7}  it follows the following fact.

\begin{corol}\label{1969}
A $\theta$-family transports each minimal spacelike surface $(M,X)$ of $\real^3_1$ to a minimal surface $(M,Y)$ of $\euclidean$. 
Moreover, the complex Gauss map is defined in $M$ by the map $\tau = \tau(p)$, $p \in M$, and the map $a$ is given by
$$a(\tau) = \frac{\tau^{1} + i \tau^{2}}{\tau^{0} + 1} \; \; \; \mbox{ with } \; \; \; 
\vert a(\tau) \vert = \sqrt{\frac{\tau^{0} - 1}{\tau^{0} + 1}\; }.$$
\end{corol}

\vspace{0.1cm}
Next corollary is consequence of the Lemma \ref{lema2.7} combined with a classic result duo to R. Osserman in \cite{7}, namely
\begin{lemma}[Osserman] The image under the Gauss map of a complete regular non plane minimal surface of the 
Euclidean space $\euclidean$ is dense in the Riemann sphere.
\end{lemma}
\begin{corol}\label{9}
If $(M,X)$ in $\real^3_1$ and $(M,Y)$ in $\euclidean$ are two minimal surfaces associated each other 
by a $\theta$-family, then these two surfaces are not complete. 
\end{corol}
\begin{proof} In fact, from the hypotheses and Lemma \ref{lema2.7}  it follows that $\nu(M)$ omits all (open and non empty) north hemisphere of the Riemann sphere. Hence it follows from the Osserman Lemma that both surfaces are not complete.
\end{proof}

In other direction, we will use the Gauss curvature formula to characterize the set of planar points of the surfaces of $\theta$-family. In fact, equation (\ref{curvGausslema2.5}) is well defined because we are assuming $\mu (w)\ne 0$ and there are not zeros of equation  $1 - 2\cos \theta \vert a \vert^{2} + \vert a \vert^{4} = 0$ since $\{w : \vert a(w) \vert = 1\} = \emptyset$.  Moreover, for all $\theta \in \mathbb R$ one has $\cos \theta \ne \frac{2|a|^2}{1+ |a|^4}$. Hence the Gauss curvature  is zero  in a point $w$ if and only if $|a'(w)|^2 =0$.
 Therefore it follows that if for a surface of the $\theta$-family whole its points are planar then all the other surfaces have also whole its points planar. More explicitly one has the following lemma. 
\begin{lemma}\label{lema3.6}
In a $\theta$-family $(U, F(\theta,w))$ the following statements are equivalents:
\begin{itemize}
\item[(1)] There exists $\theta_{0}\in\real$ such that for each $w \in U$ the curvature $K(F(\theta_{0},w)) = 0$.  
\item[(2)] The map of isothermal representation $a(w)$ is constant in $U$.
\item[(3)] The normal maps $L_{3}(a)$ and $L_{0}(a e^{i\theta})$ are constants from $U$ in the lightcone $\mathcal{C}$.
\item[(4)] For all $\theta \in \real$ the planar points  set $\mbox{Pl}(F(\theta;U)) = U.$
\end{itemize}
\end{lemma}
\begin{proof}
Immediately from (\ref{curvGausslema2.5}), we have $(1)\Leftrightarrow (2).$
Using (\ref{aplicacaoL0}) and (\ref{aplicacaoL3}), one obtains $(2)\Leftrightarrow (3).$
Finally, it notes that since $(4)\Leftrightarrow (2),$ then $(4)\Leftrightarrow (3).$
\end{proof}

\section{Application to Graph surfaces of $\mink$} 
\label{sec:3}

Through of this paper we have seen that our results are based in the existence of local isothermal coordinates in which we have the representation formula using the maps $\mu(w)$ and $a(w)$. Here we consider a particular type of graph surfaces in $\mathbb R^4_1$ and we identify the local conditions to the existence of explicit isothermal coordinates in which we can represent it as the graph of some smooth function $A(u,v)$.

\vspace{0.1cm}
It assumes the graph surface $X:U \subseteq \real^{2} \to \mink$ given by 
$$X(x,y) = (x, y, A(x,y), B(x,y))$$ 
where we  have assumed that the functions $A$, $B$ are $C^{\infty}(U,\real),$ $U$ is a connected, simply 
connected open subset of $\real^{2}$ and that the surface traced by the map $X$ is a spacelike surface of $\mink$.  We note that  if $B(x,y) \equiv 0$ one has the spacelike graph in $\real^3_1$ given by 
\begin{equation}\label{17}
X(x,y) = (x,y, A(x,y), 0) \in \real^3_1,
\end{equation}
with the induced metric tensor over $X(U)$, seen as a spacelike submanifold of $\mink.$ 

\vspace{0.2cm}
Hence if we consider  $X(x,y) = (x,y,A(x,y),0)$ to be a spacelike minimal   graph on $\real_{1}^{3}$,  its normal Gauss map $\tau$ and the normal maps  $L_0(X), 
L_3(X)$  take the form  
\begin{equation}\label{16}
\tau(x,y) = \frac{1}{\sqrt{(A_{x})^{2} - (A_{y})^{2} - 1\;}}(A_{x},-A_{y},1,0)
\end{equation}
 \begin{align}
L_{0}(X) = \left(A_{x},-A_{y},1, \sqrt{(A_{x})^{2} - (A_{y})^{2} - 1\;}\right)\label{eq33} \\ 
L_{3}(X) = \left(A_{x},-A_{y},1, -\sqrt{(A_{x})^{2} - (A_{y})^{2} - 1\;}\right).\label{eq34}
\end{align}

We observe that for obtaining the expressions (\ref{eq33}) and (\ref{eq34}) we have used the expressions for $a(w)$ in Corollary (\ref{1969}) and the expressions of $L_0(a), L_3(ae^{i\theta})$ given by (\ref{aplicacaoL3}) and (\ref{aplicacaoL0}) for $\tau$ being given by (\ref{16}).

 In this moment we note that from \cite{DuFF} one knows that a local representation of the surface $S  \subset \mathbb R^4_1$
is given by   $f_{w} = \mu \left(a + b, 1 + ab, i(1 -ab),a - b \right)$. Since we are interested when $a = b,$ one obtains the
representation
$f_{w} = \mu(2a,1 + a^{2}, i(1 - a^{2}),0).$ \ Now, let $z(w) = x(w) + i y(w)$  be a transformation of coordinate system wrote as 
\begin{equation}\label{eq36} \left\{\begin{matrix}
x = x(u,v) \\ y = y(u,v).
\end{matrix} \right.
\end{equation}
So, assuming that the graph (\ref{17}) is minimal, we can write the graph as a  minimal parametric surface $(U,f)$ with $f_{w} = \mu(2a,1 + a^{2},i(1 - a^{2}),0)$, for $a, \mu$  holomorphic functions, if the system 
\begin{equation}\label{eq37} \begin{cases}
x_{w} = 2 \mu a \\ y_{w} = \mu(1 + a^{2}) \\ A_{x} x_{w} + A_{y} y_{w} = \mu i(1 - a^{2})
\end{cases} 
\end{equation}
 is satisfied and if, additionally the Jacobian of the coordinate transformation (\ref{eq36})  is different from zero. But the latter condition is satisfied if and only if the imaginary part of  $a(w)$ is different from zero. In fact, since $$2\mu\overline{\mu}[a(1 + \overline{a}^{2}) - \overline{a}(1 + a^{2})] = 2\mu\overline{\mu}(1 + a \overline{a})(a - \overline{a}) = 4i \mu\overline{\mu}(1 - a \overline{a}) \Im(a),$$
 the Jacobian of the coordinate transformation is zero whenever 
\begin{equation}\label{eq37}
x_{w} y_{\overline{w}} - x_{\overline{w}} y_{w} = 0 \; \; \mbox{ if and only if } \; \; \Im(a) = 0.
\end{equation} 

\vspace{0.2cm}
The last fact suggests us the following lemma which follows from the Picard Theorem.

\begin{lemma}\label{lema4.8}
Let $X(x,y) = (x,y,A(x,y),0)$ be a minimal spacelike graph on $\real_{1}^{3}$ of a smooth function $A(x,y)$ defined for all points 
$(x,y) \in \complex$.  If there exists a subspace $S^{*} \subset S = (\complex,X)$  that is dense connected open subset of $S$  such that $\forall p \in S^{*}$, 
$\Im(a(p)) \neq 0,$
where $a(p)$ is the complex Gauss map associated to the surface $S = X(\complex),$ then  
$a(p)$ is a constant function. The latter means that $S$ is a spacelike plane of $\real_{1}^{3}$.    
\end{lemma}

\begin{proof}
Without loss of generality, it assumes that in $p_{0} \in S^{*}$ we have $\Im(a(p_{0})) > 0$.  By the connectedness of $S^{*}$ it follows that $\Im(a(p)) > 0$ for all $p \in S^{*}$. Since $S^{*}$ is dense in $S$ and surfaces are metrizable topological spaces, we have that for each $q \in S$, 
$\Im(a(q)) = \displaystyle\lim_{n \rightarrow \infty} \Im(a(p_{n}))$ where $\{p_{n}: n \in \mathbb{N}\} \subset S^{*}$ is a sequence that converge to $q \in S$. 
Now, by the assumption that $S^{*}$ is a dense connected open subset of $S$ and using the fact  that the projection $pr : \real^{2} \to \real$ given by $pr(x,y) = x$ is an onto continuous 
open map, we obtain $pr(S^{*}) = \real.$ So, the statement follows from the Picard Theorem.
\end{proof}
Even more, we obtain  informations about the planar points set for a general kind of minimal isothermal surface of $\real^3_1$ as follows.
\begin{prop}
Let $(M,f)$ be a connected minimal isothermal surface of $\real^3_1$ with 
\begin{equation}
f_{w} = \mu \left (2a,1 + a^{2},i(1 - a^{2}),0 \right),
\end{equation}
for $a(w), \mu (w)$ holomorphic functions. If there exists $p \in M$ such that $K(p) \neq 0,$ then the set of planar points of $(M,f)$ is a discrete set.   
\end{prop} 
\begin{proof}
First we observe that from equation (\ref{curvGausslema2.5}) for $\theta = 0$, it follows that $K(f(w)) = 0$ if and only if  $a'(w) = 0$. Therefore by hypothesis there exists a point $w_{0} = p$  where $a'(w_{0}) \neq 0$. Now, since $a'(w)$ is also a holomorphic function it follows that the set of its zeros is a discrete set, in all connected 
component of $p$. Finally, since $M$ is one of this connected components, so it follows the statement.
\end{proof}
We finish this paper establishing the next lemma related to coordinate transformation (\ref{eq36}). It shows that the restriction $\Im (a) \ne 0$ is needed condition to the existence of local isothermal coordinates  $(U, f)$ such that $f_w = \mu (2a, 1 + a^{2}, i(1 - a^{2}),0)$.

\begin{lemma}
Let $S = (U,X)$ be a minimal parametric surface of $\real_{1}^{3}$ given by $X(x,y) = (A^{1}(x,y),A^{2}(x,y),A^{3}(x,y))$. Assume there is a coordinate transformation (\ref{eq36})  which allow us to obtain locally a representation $(U,f)$ in  isothermal parameters such that 
$$f_{w} = \mu (2a, 1 + a^{2}, i(1 - a^{2}),0).$$ 
Then $\Im(a(w)) \neq 0$ for all $w \in U$. 
\end{lemma}

\begin{proof} By equation (\ref{eq37}) it follows that the Jacobian function 
$$\det\left(\frac{\partial(A^{i},A^{j})}{\partial(w,\overline{w})}\right) = \det\left(\frac{\partial(A^{i},A^{j})}{\partial(x,y)}\right) \left\vert
\begin{matrix}
x_{w} & x_{\overline{w}} \\ y_{w} & y_{\overline{w}}
\end{matrix} \right\vert$$ 
is equal to zero, for all the three possible cases where $1 \leq i < j \leq 3$, if and only if $\Im(a) = 0$.     
\end{proof}

{\bf Acknowledgments} The first author's research was supported by Projeto Tem\'atico
Fapesp n. 2016/23746-6. S\~ao Paulo. Brazil.


\begin{thebibliography}{99}


\bibitem{chern} S-S. Chern, {\it An elementary proof of the existence of isothermal parameters on a surface}, Proc. Amer. Math. Soc., 6 (1955) 771-782.

\bibitem{DPS} M.P. Dussan, A.P. Franco Filho and R.S. Santos, {\it  Spacelike Minimal surfaces which are graphs in $\mathbb R^4_1$}. J. Math. Anal. App.,  519 (2022) 126791.  https://doi.org/10.1016/j.jmaa.2022.126791


\bibitem{DuFF} M.P. Dussan, A.P. Franco Filho and P. Sim\~oes, {\it Spacelike Surfaces in $\mathbb{L}^4$ with null mean curvature vector and the nonlinear {R}iccati partial differential equation}, Nonlinear Analysis, 207  (2021) 112271.

\bibitem{3} D. Hoffman., R. Osserman., {\it The Geometry of the generalized Gauss map}, Mem. Amer. Math. Soc.
236 (1980) 1-105.

\bibitem{jccn65} J.C.C. Nitsche, {\it On new result in the theory of minimal surfaces}, Bull. of the Amer. Math. Soc. 71  (1965) 195-270. 

\bibitem{7} R. Osserman, A Survey of Minimal Surfaces, VanNostrand, NewYork, IOS 9. 



\end{thebibliography}
\end{document}